\documentclass[12pt]{amsart}

\usepackage{amssymb,latexsym, color}
\usepackage{enumerate}

\makeatletter
\@namedef{subjclassname@2010}{%
  \textup{2010} Mathematics Subject Classification}
\makeatother

\theoremstyle{plain}
\newtheorem{theorem}{Theorem}

\newtheorem{proposition}{Proposition}
\newtheorem{lemma}{Lemma}
\theoremstyle{definition}

\newtheorem{remark}{Remark}
\newtheorem{question}{Question}

\DeclareMathOperator{\red}{red}

\def\PP{\mathbb{P}^}
\def\sr{\mathrm{r}}
\def\sbr{\mathrm{\underline{r}}}

\newcommand{\C}{\mathcal{C}}
\newcommand{\T}{\mathcal{T}}
\newcommand{\Y}{\mathcal{Y}}
\newcommand{\D}{\mathcal{D}}
\newcommand{\LL}{\mathcal{L}}
\newcommand{\R}{\mathcal{R}}

\date{}

\begin{document}

\baselineskip=17pt

\title{Curvilinear schemes and maximum rank of forms}

\author[E. Ballico]{Edoardo Ballico}
\address{Dept. of Mathematics\\
  University of Trento\\
I-38123 Povo (TN), Italy}
\email{edoardo.ballico@unitn.it}

\
\author[A. Bernardi]{Alessandra Bernardi}
\address{Dept. of Mathematics, University of Bologna\\
I-40126 Bologna, Italy
}
\email{alessandra.bernardi5@unibo.it}

\date{}

\thanks{The authors were partially supported by CIRM of FBK Trento 
(Italy), Mathematical Department of Trento (Italy), Project Galaad of INRIA Sophia Antipolis M\'editerran\'ee 
(France),  Marie Curie: Promoting science (FP7-PEOPLE-2009-IEF), MIUR and GNSAGA of 
INdAM (Italy), Mathematical Department of Turin ``Giuseppe Peano'' (Italy).}

\begin{abstract} We define the \emph{curvilinear rank} of a degree $d$ form $P$ in $n+1$ variables as the minimum length  of a curvilinear scheme, contained in the $d$-th Veronese embedding of $\mathbb{P}^n$, whose span contains the projective class of $P$. Then, we give a bound for rank of any homogenous polynomial, in dependance on its curvilinear rank.
\end{abstract}


\maketitle

\section*{Introduction}

The \emph{rank}  $\sr(P)$ of a homogeneous polynomial $P\in \mathbb{C}[x_0, \ldots , x_n]$ of degree $d$, is the  minimum $r\in \mathbb{N}$ such that $P$  can be written as sum of $r$ pure powers of linear forms $L_1, \ldots, L_r \in \mathbb{C}[x_0, \ldots , x_n]$: 
\begin{equation}\label{FT}
P=L_1^d+ \cdots + L_r^d.
\end{equation} 

A very interesting open question is to determine the maximum possible value that the rank of a form (i.e. a homogeneous polynomial) of given degree in a certain number of variables can have.
\\
On our knowledge, the best general achievement on this problem is due to J.M. Landsberg and Z. Teitler that in \cite[Proposition 5.1]{lt} proved that  the rank of a degree $d$ form in $n+1$ variables  is smaller than or equal to
$
{n+d \choose d}-n
$
.
Unfortunately this bound is sharp only for $n=1$ if $d\geq 2$; in fact, for example, if $n=2$ and $d=3,4$, then the maximum ranks are 5 and 7 respectively (see \cite[Theorem 40 and 44]{bgi}).
More recently G. Blekherman and Z. Teitler proved in \cite{bt} that the maximum rank is always smaller than or equal to twice the generic rank that is the rank of a generic polynomial, i.e. the minimum $r$ s.t. the $r$-th secant variety to Veronese variety fills up the ambient space (such a secant variety is classicaly defined to be the Zariski closure of the set of all $r$-th secant space to a Veronesean). In the celebrated Alexander and Hischowitz paper \cite{ah} they computed the dimensions of all such secant varieties, so the generic rank is nowadays considered a classical result. Cleary finding a bound for the rank of any polynomial given the number of variables and the degree is a very different and difficult problem.

Few more results were obtained by focusing the attention on limits of forms of given rank. When a form $P$ is in the Zariski closure of the set of forms of rank $s$, it is said that $P$ has \emph{border rank} $\sbr(P)$ equal to $s$. For example,  the maximum rank of forms of border ranks 2, 3 and 4 are known (see \cite[Theorems 32 and 37]{bgi} and \cite[Theorem 1]{bb2}).
In this context, in \cite{bb1} we posed the following:

\begin{question}[\cite{bb1}]\label{q1}
Is it true that $\sr(P) \le d(\sbr(P)-1)$ for all degree $d$ forms $P$? Moreover, does the  equality hold if and only if the projective class of $P$ belongs to the tangential variety of a Veronese variety?
\end{question}

The Veronese variety $X_{n,d}\subset \mathbb{P}^{N_{n,d}}$, with $n \ge 1$,  $d \ge 2$ and 
$N_{n,d}:= {{n+d}\choose d}-1$ is the classical  $d$-uple Veronese embedding $\nu_d: \PP n \to \PP 
{N_{n,d}}$ and parameterizes projective classes of degree $d$ pure powers of linear forms in $n+1$ variables. Therefore the rank $\sr(P)$ of  $[P]\in \mathbb{P}^{N_{n,d}}$ is the minimum $r$ for which there exists a  length $r$ smooth zero-dimensional scheme $Z\subset X_{n,d}$ whose span contains $[P]$ (with an abuse of notation we are extending the definition of rank of a form $P$ given in (\ref{FT}) to its projective class $[P]$). More recently, other notions of polynomial rank have been introduced and widely discussed (\cite{bubu}, \cite{rs}, \cite{br}, \cite{bbm},   \cite{bb3}). They are all related to the minimal length of a certain zero-dimensional schemes  embedded in $X_{n,d}$ whose span contains the given form. Here we recall only the notion of \emph{smoothable rank} $\mathrm{smr}(P)$ of a form $P$ with $[P]\in \mathbb{P}^{N_{n,d}}$ (see \cite{bbm,bubu}):
$$\mathrm{smr}(P)=\min \left\{\deg (Z) \; | \;  Z \hbox{ limit of smooth schemes } Z_i, \deg(Z_i)=\deg(Z), \right.$$
$$\left. Z,Z_i\subset X_{n,d}, \, \dim_KZ=\dim_K Z_i =0 \hbox{ and } [P]\in \langle Z \rangle\right\}.$$
With this definition, it seems more reasonable to state Question \ref{q1} as follows:
\begin{question}\label{q2}
Fix $[P]\in \mathbb {P}^{N_{n,d}}$. Is it true that  $\sr (P) \le (\mathrm{smr}(P)-1)d$~?
\end{question}

In this paper we want to deal with a more restrictive but easier to handle  notion of rank, namely the ``curvilinear rank". 
We say that a scheme $Z\subset \mathbb{P}^N$ is \emph{curvilinear} if it is a finite union of schemes of the form $\mathcal{O}_{\C_i,P_{i}}/{\mathfrak{m}}_{P_i}^{e_i}$ for smooth points $P_i$ on reduced curves $\C_i\subset \mathbb{P}^N$, or equivalently that the tangent space at each connected component of $Z$ supported at the $P_i$'s has Zariski dimension $\leq 1$. We define the \emph{curvilinear rank} $\mathrm{Cr}(P)$ of a degree $d$ form $P$ in $n+1$ variables as:
$$\mathrm{Cr}(P):=\min\left\{\deg(Z)\; | \; Z\subset X_{n,d}, \; Z \hbox{ curvilinear, } [P]\in \langle Z \rangle\right\}.$$
The main result of this paper is the following:

\begin{theorem}\label{i1}
For any degree $d$ form $P$ we have that
$$\sr (P) \le (\mathrm{Cr}({P})-1)d + 2-\mathrm{Cr}({P}) .$$
\end{theorem}

Theorem \ref{i1} is sharp if $\mathrm{Cr}(P) =2, 3$ (\cite[Theorem 32 and 37]{bgi}).

Clearly if a scheme is curvilinear is also smoothable, so the next question will be to understand if Theorem \ref{i1} holds even though we substitute the curvilinear rank with the smoothable rank:

\begin{question}\label{i2}
Fix $[P]\in \mathbb {P}^{N_{n,d}}$. Is it true that $\sr (P) \le (\mathrm{smr}(P)-1)d +2 -\mathrm{smr}(P)$~?
\end{question}

This manuscript is organized as follows: Section \ref{sec1} is entirely devoted to the proof of Theorem \ref{i1} with a lemma; in Section \ref{sec2} we study the case of ternary forms and we prove that, in such a case, Question \ref{q2} has an affirmative answer. 

We will always work with an algebraically closed field $K$ of characteristic 0. 

\section{Proof of Theorem \ref{i1}}\label{sec1}
Let us begin this section with a Lemma that will allow us to give a lean proof of the main theorem.

We say that an irreducible curve $\T$ is {\it rational} if its normalization is a smooth rational curve.
\begin{lemma}\label{a3}
Let $Z\subset \mathbb {P}^r$, $r\ge 2$, be a zero-dimensional curvilinear scheme of degree $k$. Then
there is an irreducible and rational curve $\T\subset \mathbb {P}^r$ such that $\deg (\T) \le k-1$
and $Z\subset \T \subseteq \langle Z\rangle$.
\end{lemma}

\begin{proof} 
If the scheme $Z$ is in linearly general position, namely $\langle Z \rangle \simeq \mathbb{P}^{k-1}$, then there always exists a rational normal curve of degree $k-1$ passing through it (this is a classical fact, see for instance \cite[Theorem 1]{eh}). 
If $Z$ is not in linearly general position, consider $\mathbb{P}(H^0(Z,\mathcal {O}_Z(1)))\simeq \mathbb{P}^{k-1}$. In such a $\mathbb{P}^{k-1}$ there exists a curvilinear scheme $W$ of degree $k$ in linearly general position such that the projection $\ell _V\colon \mathbb {P}^{k-1} \setminus V \to \langle Z\rangle$ from a $( k-\dim (\langle Z\rangle )-2)$-dimensional vector space $V$ induces an isomorphism between $W$ and $Z$. Consider now the degree $k-1$ rational normal curve $\C\subset \mathbb{P}^{k-1}$ passing through $W$, its projection $\ell_V(\C)$ contains $Z$ and it is irreducible and rational since $\C$ is irreducible and rational and, by construction, $\deg(\ell_V(\C))\leq \deg(\C)=k-1$.
\end{proof}
We do not claim that the curve $\T$ is smooth, because we only need that its normalization is $\mathbb {P}^1$.

Let $X\subset \mathbb {P}^r$ be an integral non-degenerate variety. For any $P\in \langle X\rangle$ the \emph{$X$-rank}  $r_X(P)$ is the minimal cardinality of a subset
$S\subset X$ such that $P\in \langle S\rangle$.

We are now ready to prove the main theorem of this paper.
\\
\\
\indent {\emph {Proof of Theorem \ref{i1}:}} Let $X_{n,d}$ be the Veronese image of $\mathbb{P}^n$ into $\mathbb{P}^{{n+d \choose d}-1}$ via $\mathcal{O}(d)$, let $Z\subset X_{n,d}$ be a minimal degree curvilinear scheme such that $P\in \langle Z \rangle$, and let $U\subset \mathbb{P}^n$ be the curvilinear scheme such that $\nu_{d}(U)=Z$.
The minimality of $Z$ gives $P\notin \langle Z'\rangle$ for any $Z'\subsetneq Z$. Say that  $\mathrm{Cr}(P)=\deg(Z)=\deg(U):=k\ge 2$. If $Z$ is reduced, then $\sr ({P}) =k$ and the statement of the theorem in this case is trivial. Hence we may assume that $Z$ is not reduced.
By Lemma \ref{a3}, there exists a rational curve $\T\subset \mathbb {P}^n$
 such that $U\subset \T$ and $c:= \deg (\T)\le k-1$. Set $\Y:= \nu _d(\T)$:
 $$\begin{array}{rcl}\mathbb{P}^n&\stackrel{\nu_d}{\hookrightarrow}&\mathbb{P}^{{n+d\choose d}-1}\\
U \subset \T&\mapsto& Z\subset \Y
\end{array}.$$
 The curve $\Y\subset \mathbb {P}^{N_{n,d}}$ has degree $cd$ and $Z\subset \Y$. Hence
 $P\in \langle \Y\rangle$. Since $\Y\subset X_{n,d}$, we have $\sr ({P})\le r_{\Y}({P})$. Hence it is sufficient to prove
 that $r_\Y({P}) \le d(k-1)+2-k$. Since the function $t\mapsto dt$ is increasing and $c\le k-1$, it is sufficient to prove that
 $r_\Y({P}) \le dc+2-k$. Since $\T$ is a degree $c$ rational curve, there are a rational normal curve $\D\subset \mathbb {P}^c$
such that $\T$ is obtained from $\D$ using the linear projection from a linear subspace $E\subset \mathbb {P}^c$ with $\dim (E) = c-\dim (\langle E\rangle )-1$ and $E\cap \D =\emptyset$. 
We use the embedding $\nu _d$ also for any projective
space. We need to use it for $\mathbb {P}^s$ with $s:= \max \{n,c\}$. Now let $\C:= \nu _d(\D)$. 
 $$\begin{array}{rcl}\mathbb{P}^c&\stackrel{\nu_d}{\hookrightarrow}&\mathbb{P}^{{c+d\choose d}-1}\\
 \D&\mapsto&\C\\
 &&\\
 \downarrow & &\downarrow \ell_M \\
 &&\\
 \mathbb{P}^n&\stackrel{\nu_d}{\hookrightarrow}&\mathbb{P}^{{n+d\choose d}-1}\\
\T&\mapsto& \Y
\end{array}.$$

The curve $\C$ is a degree $cd$ rational normal curve in its linear span $\langle \C\rangle \cong \mathbb {P}^{dc}$. 
Since $\Y$ is embedded
in $\mathbb {P}^{N_{n,d}}$ by the restriction of the degree $d$ forms, $\Y$ is a linear projection of $\C$ from a linear  subspace $M\subset \mathbb {P}^{dc}$ such
that $\C\cap M =\emptyset$ and $\dim (M) = cd -\dim (\langle \Y\rangle )-1$ (we have $M\cap \C =\emptyset$, because $\deg (\Y) =cd$). Call $\ell _M\colon  \mathbb {P}^{dc} \setminus M\to
\langle \Y\rangle$ the linear projection from $M$. Since $\C\cap M=\emptyset$, the morphism $\ell _M$ is surjective. Since $M\cap \C =\emptyset$,
the map $\ell _M|_\C$ is a degree one morphism $\ell : \C \to \Y$. Set $W:= \ell ^{-1}(Z)$ (scheme-theoretic counterimage). Since $\ell$ is proper and surjective,
$\ell (W) =Z$ and hence $\deg (W) = k$. 
 $$\begin{array}{rcl}\mathbb{P}^c&\stackrel{\nu_d}{\hookrightarrow}&\mathbb{P}^{{c+d\choose d}-1}\\
 \D&\mapsto&W\subset\C\\
 &&\\
 \downarrow & &\downarrow \ell_M \\
 &&\\
 \mathbb{P}^n&\stackrel{\nu_d}{\hookrightarrow}&\mathbb{P}^{{n+d\choose d}-1}\\
U\subset\T&\mapsto& Z\subset\Y
\end{array}.$$


Set $\ell ':= \ell _M|(\langle W\rangle \setminus M\cap \langle W\rangle)$ and notice that even though by construction we clearly have that $W\cap M=\emptyset$, we cannot assume that also $M\cap \langle W\rangle=\emptyset$. Since $\ell (W) =Z$ and $\ell _M$ is surjective,
$\ell '$ is surjective. Fix $O\in \langle W\rangle \setminus M\cap \langle W\rangle$ such that $\ell '(O) =P$. Since $P\notin \langle Z'\rangle$ for each $Z'\subseteq Z$
and $W = \ell ^{-1}(Z)$, then $O\notin \langle W'\rangle$ for any $W'\subsetneq W$.

\quad (a) First assume $\deg (W) \le \lfloor(dc+2)/2\rfloor$. This implies that $O$ has border rank $\deg (W)$ and that either $\mathrm{r}_\C(O) =\deg (W)$ or $\mathrm{r}_\C(O) = dc+2-\deg (W)$ (\cite{cs}, \cite[Theorem 4.1]{lt}, \cite[Theorem 23]{bgi}). Take $S\subset \C$ evincing $\mathrm{r}_\C(O)$. Since $P =\ell _M(O)$, we have
$P\in \langle \ell (S)\rangle$. Since $\sharp (\ell (S)) \le \sharp (S) \le cd+2-k$, we get $\mathrm{r}_\Y({P})\le cd+2-k$.

\quad (b) Now assume $\deg (W) > \lfloor(dc+2)/2\rfloor$. A classical result attributed to Sylvester gives the relation between the length of two 0-dimensional subschemes contained in the rational normal curve and such that their spans contain the same point (see e.g. \cite{bgi, cs}). If $P\in \langle A \rangle\cap \langle B \rangle$ with $A,B$ two 0-dimensional schemes on the rational normal curve of degree $d$ then $\deg(A)+\deg(B)=d+2$. Since $P\notin \langle W'\rangle$ for any $W'\subsetneq W$ and
any zero-dimensional subscheme of $\C$ with degree at most $dc+2$ is linearly independent, Sylvester's theorem gives $\mathrm{r}_\C(O) \le \deg (W)$. As in step (a) we get
$\mathrm{r}_\Y({P})\le k < d(k-1)+2-k$.
\qed

\section{Superficial case}\label{sec2}

In this section we show that Question \ref{q2} has an affirmative answer in the case $n=2$ of ternary forms and that the bound in Question \ref{q2} is seldom sharp in this case
(for large $\mathrm{cr}({P})$ the upper bound in Question \ref{q2} is worst than the one true by \cite{bt}).
 
More precisely, we prove the following result.

\begin{proposition}\label{d1}

Let $P$ be a ternary form of degree $d$ with $2 \le \mathrm{cr}({P}) \le d$.
If $\mathrm{cr}(P) \le d$, then $\sr (P) \le \binom{d+2}{2} -\binom{d-\mathrm{cr}({P})+1}{2}-1$.
\end{proposition}

Before giving the proof of Proposition \ref{d1}, we need the following result.

\begin{proposition}\label{d2}

Let $Z\subset \mathbb {P}^2$ be a degree $k\ge 4$ zero-dimensional scheme. 

There is an integral curve $\C\subset \mathbb {P}^2$
such that $\deg (\C) = k-1$ and $Z\subset \C$ if and only if $Z$ is not contained in a line.

\end{proposition}

\begin{proof}
First assume that $Z$ is contained in a line $\D$. B\'ezout theorem gives that $\D$ is the only integral curve of degree $<k$ containing $Z$.

Now assume that $Z$ is not contained in a line. 

\quad {\emph {Claim 1.}} The linear system $\vert \mathcal {I}_Z(k-1)\vert$ has no base points
outside $Z_{\red}$.

\quad {\emph {Proof of Claim 1.}} Fix $P\in \mathbb {P}^2\setminus Z_{\red}$. Since
$\deg (Z\cup \{P\}) = k+1$, we have $h^1(\mathcal {I}_{Z\cup \{P\}}(k-1)) >0$ if
and only if there is a line $\D$ containing $Z\cup \{P\}$, but, since in our case $Z$ is not contained in a line, we
get $h^1(\mathcal {I}_{Z\cup \{P\}}(k-1))=0$. Hence $h^0(\mathcal {I}_{Z\cup \{P\}}(k-1))
= h^0(\mathcal {I}_Z(k-1))-1$, i.e. $P$ is not a base point of $\vert \mathcal {I}_Z(k-1)\vert$.

By Claim 1, the linear system $\vert \mathcal {I}_Z(k-1)\vert$ induces a morphism $\psi \colon \mathbb {P}^2\setminus Z_{\red} \to \mathbb {P}^x$.

\quad {\emph {Claim 2.}} We have $\dim (Im (\psi) ) =2$.

\quad {\emph {Proof of Claim 2.}} It is sufficient to prove that the differential $d\psi (Q)$
of $\psi$ has rank $2$ for a general $Q\in \mathbb {P}^2$. Assume that $d\psi (Q)$ has rank $\le 1$, i.e.
assume the existence of a tangent vector ${\bf {v}}$ at $Q$ in the kernel of the linear
map $d\psi (Q)$. Since $h^1(\mathcal {I}_{Z\cup \{P\}}(k-1)) =0$ (see proof of Claim 1), this
is equivalent to $h^1(\mathcal {I}_{Z\cup {\bf {v}}}(k-1)) >0$. Since $\deg (Z\cup {\bf {v}})=k+2
\le 2(k-1)+1$, there is a line $\D\subset \mathbb {P}^2$ such that $\deg (\D\cap (Z\cup {\bf {v}})) \ge k+1$ (\cite[Lemma 34]{bgi}). Hence $\deg (Z\cap \D) \ge k-1$. Since $k\ge 4$ there are at most finitely many lines $\D_1,\dots ,\D_s$ such that $\deg (\D_i\cap Z)\ge k-1$ for all $i$. If $Q\notin \D_1\cup \cdots \cup \D_s$, then $\deg (\D\cap (Z\cup {\bf {v}}))\le k$ for every line $\D$.

By Claim 2 and Bertini's second theorem (\cite[Part 4 of Theorem 6.3]{j}) a general
$\C\in \vert \mathcal {I}_Z(k-1)\vert$ is irreducible.
\end{proof}

Any degree $2$ zero-dimensional scheme $Z\subset \mathbb {P}^n$, $n\ge 2$ is contained in a unique line and hence it is contained in a unique irreducible curve of degree $2-1$. Now we check that in case our form has curvilinear rank equal to $3$, then Proposition \ref{d2} fails in a unique case.

\begin{remark}\label{d3}
Let $Z\subset \mathbb {P}^2$ be a zero-dimensional scheme such that
$\deg (Z)=3$. Since $h^1(\mathcal {I}_Z(2))=0$ (\cite{bgi}, Lemma 34), we
have $h^0(\mathcal {I}_Z(2))=3$. A dimensional count gives that
$Z$ is not contained in a smooth conic if and only if there is $P\in \mathbb {P}^2$ with
$Z =2P$ (in this case $\vert \mathcal {I}_Z(2)\vert$ is formed by the unions $\R\cup \LL$ with
$\R$ and $\LL$ lines through $P$).
\end{remark}

We conclude our paper with the Proof of Proposition \ref{d1}.

\qquad {\emph {Proof of Proposition \ref{d1}.} Take $Z\subset \mathbb {P}^2$ evincing the cactus rank. If $Z$ is contained in a line $\LL$, then $P\in \langle \nu _d(\LL)\rangle$
and hence $\sr ({P})\le d$ by a theorem of Sylvester already recalled in item (b) of the proof of our mail theorem (\cite{bgi, cs}) or by \cite{lt}, Proposition 5.1.
Now assume that $Z$ is not contained in a line. Let $\C\subset \mathbb {P}^2$ be an integral curve of degree $\mathrm{cr}({P})-1$
containing $Z$. We have $P\in \langle \nu _d({\C})\rangle$ and $\dim (\langle \nu _d({\C}))\rangle = \binom{d+2}{2} -\binom{d+1-k}{2} -1$.
Apply \cite{lt}, Proposition 5.1.
\qed

\providecommand{\bysame}{\leavevmode\hbox to3em{\hrulefill}\thinspace}

\end{document}